\documentclass[11pt]{amsart}

\usepackage[margin=1.25in]{geometry}
\usepackage{amscd,amssymb, amsmath, setspace, graphicx,color,mathtools}
\usepackage{url}
\usepackage{hyperref}
\usepackage[utf8]{inputenc} 
\usepackage{float}
\usepackage{verbatim}
\usepackage{tikz}

\makeatletter
\@namedef{subjclassname@1991}{2020 Mathematics Subject Classification}
\makeatother

\DeclareMathOperator{\rlct}{RLCT}

\DeclareMathOperator{\jac}{Jac}

\numberwithin{equation}{section}

\theoremstyle{plain}
\newtheorem{theorem}{Theorem}[section]

\newtheorem{conj}[theorem]{Conjecture}
\newtheorem{lemma}[theorem]{Lemma}

\newtheorem{proposition}[theorem]{Proposition}

\theoremstyle{definition}
 
\newtheorem{definition}[theorem]{Definition}
\newtheorem{remark}[theorem]{Remark}

\newtheorem*{ex*}{Example}

\newtheorem{example}[theorem]{Example}

\newcommand{\Z}{\mathbb{Z}}
\newcommand{\C}{\mathbb{C}}
\newcommand{\R}{\mathbb{R}}
\newcommand{\Q}{\mathbb{Q}}

\newcommand{\gt}{>}

\title[Calculation errors and singularities]{Estimation of calculation errors \\ and resolution of singularities} 

\author[V. Fadinger-Held]{Victor Fadinger-Held}
\address{Mittelschule Bad Waltersdorf, Austria}
\email{victor.fadinger@ms-badwaltersdorf.at}
\author[D. Windisch]{Daniel Windisch}
\address{Department of Computer Science, KU Leuven, Belgium}
\email{daniel.windisch.math@gmail.com}

\begin{document}

\begin{abstract}
There are two possible answers to ``$(a+b)^2 = a^2 + b^2$''. Either ``this is completely wrong'' or ``I assume, you are in characteristic $2$''. If we go for the first answer, we could ask ourselves: ``How wrong is it actually?'' 
In this expository article, 
% we highlight a close connection between the asymptotics of the error $\vert f(x) - g(x)\vert$, where $f$ and $g$ are polynomials in several indeterminates with real coefficients, and resolution of singularities over real numbers. More precisely, 
we introduce and use resolution of singularities and the real log canonical threshold (RLCT) as a measure for such errors. The RLCT is an invariant with connections to the minimal model program and, as recently discovered and investigated, to Bayesian statistics and machine learning. Inspired by typical mistakes from high school mathematics, we compute this invariant for several classes of examples, among them the error arising from the ``freshman's dream'' above and from a flawed computation of means in elementary statistics. In some examples, we highlight how a computer algebra software can be used in order to solve this problem.
\end{abstract}


\maketitle

\section{When the right solution is still wrong}\label{sec:school}

Let $f$ and $g$ be two polynomials in $d$ variables with coefficients from the field $\R$ of real numbers and let $a \in \R^d$. We ask the following basic question: \emph{To what extend does a small change $a + \delta$ affect the difference $\vert f(a) - g(a) \vert$ and change it to $\vert f(a+\delta) - g(a+\delta) \vert$?}

To make the question more precise, set $c = f(a) - g(a)$ and $\tilde{g} = g - c$. Then $f(a) - \tilde{g}(a) = 0$, so we can assume without loss of generality that $f(a) - g(a) = 0$. Now, we might alternatively ask: \emph{For ``how many'' $x$ in a neighborhood of $a \in \R^d$ with $f(a) = g(a)$ is the distance $\vert f(x) - g(x) \vert$ ``small''?} 

For instance, if $d = 2$, $a = 0 \in \R^2$ and the function $f-g$ happens to be equal to $x^2 + y^2$, then a point $(x,y) \in \R^2$ satisfies $\vert f(x,y) - g(x,y) \vert \leq \varepsilon$ if and only if $(x,y)$ is an element of the disc with center $0$ and radius $\sqrt{\varepsilon}$. Therefore, the area of this disc, which is $\pi\cdot \varepsilon$, is a good measure for ``how many'' $(x,y)$ lead to ``small'' differences between $f$ and $g$. Since we are interested in small changes of $a$, we consider the asymptotics of this area as $\varepsilon \to 0$.

To go to higher dimensions, we fix a compact neighborhood $W \subseteq \R^d$ of $a$ and consider the $d$-dimensional Lebesgue measure on $\R^d$ which we denote by $\operatorname{vol}$. Our object of interest is the set
\[
T(\varepsilon) = \{x \in W \mid \vert f(x) - g(x)\vert \leq \varepsilon\}
\]
and we want to understand the asymptotics $\operatorname{vol} T(\varepsilon)$ as $\varepsilon \to 0$. It can be shown that $\operatorname{vol} T(\varepsilon)$ is asymptotically equal
to $C \cdot \varepsilon^\lambda (-\log \varepsilon)^{m-1}$ for some $C \in \R_{\gt 0}$ and $(\lambda,m) \in \Q_{\gt 0} \times \Z_{\gt 0}$. Therefore, the error $ \vert f(x) - g(x)\vert $ can be considered small if and only if $\lambda$ is small and $m$ is large, see also Remark~\ref{remark:asymptotics}. The pair $(\lambda,m)$ is called the \emph{real log canonical threshold} (RLCT) of $f-g$ at $a$ and was studied towards asymptotics of $\operatorname{vol} T(\varepsilon)$ in~\cite{Lin:tubes}.  Furthermore, this birational invariant of $f-g$ appears in several other contexts of pure and applied mathematics. Two most fascinating connections arise from the minimal model program and from Bayesian statistics, respectively: 
\begin{enumerate}
	\item The complex analogue of the RLCT, called the \emph{log canonical threshold} (LCT), is a measure for multiples of divisors on smooth complex varieties to be log canonical. Log canonical divisors are the ones for which minimal models are expected to exist in the minimal model program~\cite[Chapter 2]{Kollar_Mori_1998}. Moreover, Shokurov's ACC conjecture on LCTs~\cite{Shokurov1992}, which is now a theorem~\cite{ACCsolved2014}, has a direct connection to \emph{termination of flips}, one of the two open questions within the minimal model program. Indeed, the ACC conjecture yields a partial inductive proof of termination of flips. In alignment with this, we will explain in Section~\ref{sec:resolution} how to compute the RLCT by using resolution of singularities and apply this approach in the following sections.
	\item In the context of Bayesian statistic and, more precisely, in \emph{singular learning theory}~\cite{watanabe:book, Watanabe2018BayesianStatistics}, the RLCT (also called \emph{learning coefficient} in this context) determines the large sample asymptotics of the marginal likelihood of a statistical model with respect to a given data set. In the context of machine learning, this can be used to study learning and generalization behavior of models~\cite{DeepLearningIsSingular,JiayiGrokking}. In data science, model selection techniques employing the RLCT can be used for structural analysis of high-dimensional data~\cite{sbic:2017}. See~\cite{aoyagi2025neuralnetworks, drton2026singularlearningtheoryfactor} for recent results on learning coefficients of statistical models.
\end{enumerate}

For us, a starting point to think about error estimation through the lens of RLCTs was the following example from high school mathematics.

\begin{example}[Average fuel consumption]\label{example:fuel-consumption}
\emph{Mr. Smith keeps a record of the distance he drives between fill-ups and the amount of petrol he purchases. An excerpt from his fuel log is shown below.}
\begin{table}[H]
\centering
\begin{tabular}{c|c}
\textbf{Distance Driven (km)} & \textbf{Petrol Purchased (L)} \\
\hline
519 & 41.3 \\
555 & 42.2
\end{tabular}
\end{table}
\emph{Based on the data in the table, how many kilometers does Mr. Smith’s car travel per liter of petrol, on average? Round your answer to the nearest hundredth.}

Using the correct solution of first summing up the individual columns and then forming their quotient gives $12.86$. But also the solution suggested by a student of first forming the quotients in each row and then computing their arithmetic mean yields this number. The following question arises: \emph{Is this also a correct computation path?}

A quick check with other numbers leads to a clear \emph{no}. Conceptually, the second method gives the same weight to every summand in the arithmetic mean computation, which is of course flawed, but in this specific case pretty accurate because the values of the purchased petrol amounts are close to each other. 

More geometrically, we have a hypersurface defined on an open subset of $\R^2 \times \R^2$ given by
\[
\frac{x_1 + x_2}{y_1 + y_2} - \frac{1}{2} \cdot \left( \frac{x_1}{y_1} + \frac{x_2}{y_2} \right) = 0
\]
and the point $((519, 555),(41.3,42.2))$ lies sufficiently close to it. The natural question arises what ``sufficient'' means as we vary the point.

For later reference, we lift this example to the more general setting of columns of length $n$ in the table. This leads to the equation of rational functions 
\[
\frac{1}{n} \sum_{i = 1}^n \frac{x_i}{y_i} = \frac{\sum_{i = 1}^n x_i}{\sum_{i = 1}^n y_i}
\]
in variables $x_1,\ldots,x_n,y_1,\ldots,y_n$ that are evaluated in positive real numbers. The equation is equivalent to the following polynomial vanishing:
\[
p_n = \left( \sum_{i = 1}^n x_i \prod_{ i \neq j = 1 }^n y_j \right) \cdot \left( \sum_{i = 1}^n y_i \right)- n\cdot \left( \sum_{i = 1}^n x_i \right) \cdot \prod_{ j = 1 }^n y_j,
\]
which specializes to 
\[
p_2 = (x_1y_2 + x_2y_1)(y_1 + y_2) - 2(x_1+x_2)y_1y_2
\]
in the case $n = 2$.
\end{example}

Another classical mistake in high school mathematics is the so-called ``freshman's dream''. It constitutes the belief that $(x+y)^2$ equals $x^2 + y^2$ as polynomials or, equivalently, for all $x,y \in \R$.

\begin{example}[Freshman's dream]\label{example:freshman}
    Let $f = (x+y)^2 = x^2 + 2xy + y^2$ and $g = x^2 + y^2$. Then $f(x,y) = g(x,y)$ if and only if $xy = 0$, that is, $x = 0$ or $y = 0$. We will study error estimation for $f-g = 2xy$ in Example~\ref{example:easy-log-res}(2) and generalize this to higher dimensions by considering $(x_1 + \cdots + x_d)^2$ in Section~\ref{sec:freshman}.
\end{example}

The manuscript is structured as follows: In Section~\ref{sec:formalization}, we formalize some ideas around the problem and introduce the RLCT. In Section~\ref{sec:resolution}, we lay out the relation between resolution of singularities and error estimation, review some useful results, and give a list of examples for illustration. Section~\ref{sec:freshman} contains a full study of RLCTs for what we call the ``freshman's dream variety'', the real algebraic hypersurface defined by $(x_1 + \cdots+ x_n)^2 - (x_1^2 + \cdots + x_n^2)$. In the final Section~\ref{sec:petrol}, we showcase how to use computer algebra software to attack the problem of determining the RLCT. We purposely put this section at the end so that the reader should be able to follow and even code along using the knowledge from the previous sections. Namely, we use \emph{Macaulay2} to compute the RLCT of the ``petrol variety'' arising from Example~\ref{example:fuel-consumption}. The approach is valid for any dimension $n$. Our computational capacities allow us to determine the RLCT for $n \in \{2,3,4\}$ and to conjecture the general value.

\section{How wrong can a calculation path be?}\label{sec:formalization}

As illustrated by the petrol example in Section~\ref{sec:school}, the result of a calculation path for a computational problem can be correct while the calculation path itself is conceptually wrong. In this section, we will formalize what we mean by a calculation path and when we would consider it correct in our context. We will then introduce a way of measuring to which extend a small change of the input numbers of the example impacts the calculation error along a wrong calculation path. 

\subsection{What is a calculation path?} 
The structure of most basic computational problems, in particular, in high school mathematics is the following: One is given real numbers $a_1,\ldots,a_d \in \R$ and, in order to derive the correct solution, one applies a chain of manipulations to these numbers. These manipulations are -- in the vast majority of cases -- analytic functions. For instance, in Example~\ref{example:fuel-consumption}, we first form the sums $\sum_{i = 1}^n x_i$ and $\sum_{i = 1}^n y_i$ and then take their quotient, which is even a rational function. However, analytic functions like $\sin$, $\cos$ and $\exp$ or their (local) inverses frequently come up in higher grades.

To formalize this idea of an intended calculation path, we are given a chain of functions 
\begin{align*}
    &f_0(x_1,\ldots,x_d), \\
    &f_1(x_1,\ldots,x_d,z_1),\\
    &f_2(x_1,\ldots,x_d,z_1,z_2),\\
    & \ \ \  \vdots\\
    &f_\ell(x_1,\ldots,x_d,z_1,\ldots,z_\ell)
\end{align*}
and we want $f_0$ to be analytic at $a = (a_1,\ldots,a_d)$, $f_1$ to be analytic at $v_0 = ({a}, f_0({a}))$, $f_2$ to be analytic at $v_1 = ({a}, f_0 ({a}),f_1(v_0))$, and so on.
The correct result is then $f_\ell(v_{\ell-1})$. As this inductive process amounts to recursively plugging power series into power series, we can view this as a single function $f$ that is analytic in an open neighborhood of $a$. This leads us to the following:

\begin{definition}\label{definition:calculation-problem}
    A (\emph{high school}) \emph{mathematics exercise} is a pair $(a,f)$, where $a \in \R^d$ is a real vector and $f$ is a real, scalar-valued function that is analytic at $a$. We call $f$ the \emph{intended solution path}  and $f(a)$ the \emph{solution} of $(a,f)$.

    Another real, scalar-valued function $g$ that is analytic at $a$ is called a \emph{potential solution path} if $g(a) = f(a)$.
\end{definition}

Most of this paper will be about determining the asymptotics of the difference between an intended solution path and a potential solution path of an exercise.

\begin{example}\label{example:alternative-solution-path}
    In Example~\ref{example:fuel-consumption}, the function
    \[
    f = \frac{\sum_{i = 1}^n x_i}{\sum_{i = 1}^n y_i}
    \]
    is the intended solution path while 
    \[
     g = \frac{1}{n} \sum_{i = 1}^n \frac{x_i}{y_i} 
    \]
    is a potential solution path.
\end{example}

\begin{comment}
\begin{remark}\label{remark:ambiguity}
    There is some ambiguity in the definition of an alternative solution path $g$. It is, of course, possible that the exercise suggests implicitly that the intended solution $f$ should be evaluated only along some subspace $X \subseteq \R^d$ containing $a$ in which case the neighborhood $U$ of Definition~\ref{definition:calculation-problem} should be an open subset of $X$.\footnote{For instance, the exercise could be to compute the area of a square with edge length $x$. A student might first go to the formula $f(x,y) = x\cdot y$ for the area of a rectangle, but is supposed to evaluate $f$ only along the set $X = \{(x,y) \in \R_{\gt 0}^2 \mid x = y\}$.} Such an $X$ could be a linear subspace or, more generally, a smooth space defined by the vanishing of some analytic functions around $a$. This is, however, the most general that we could imagine in high school mathematics. In any case, such a smooth subspace can itself be analytically transformed to a real vector space of dimension less or equal to $d$, see Lemma~\ref{lemma:transversal}. In particular, the assumption that $U \subseteq \R^d$ is open in Definition~\ref{definition:calculation-problem} means no loss of generality.
\end{remark}
\end{comment}

\subsection{Error measurement for potential calculation paths}
In Example~\ref{example:fuel-consumption}, we were given the intended solution $f$ and a potential solution $g$ of an exercise $(a,f)$ and we wanted to figure out how bad a small change $a + \delta$ of $a$ effects the correctness of $g(a + \delta)$ as opposed to the intended value $f(a + \delta)$. One way of phrasing this is to set an error bound $\varepsilon > 0$ and to ask ``for how many values $w$ in a small neighborhood of $a$'' we are within that bound. More precisely, let $W\subseteq \R^d$ be an open neighborhood of $a$ with compact closure on which $f$ and $g$ are analytic. Write $p = f-g$. We consider the volume $\operatorname{vol} T_p(\varepsilon)$, that is, the standard $d$-dimensional Lebesgue measure of the set
\[
T_{p}(\varepsilon) =\{ w \in W \mid \vert p(w)\vert \leq \varepsilon \}.
\]
Of course, this volume depends heavily on the choice of $W$. The dependence is, however, a straightforward one that does not affect our perspective, as we will see now.

Since we are interested in small errors on $a$, it is convenient to ask for the asymptotics of $\operatorname{vol} T_p(\varepsilon)$ as $\varepsilon \to 0$. Moreover, we might choose $W$ arbitrarily small. The smaller this asymptotics, the worse the potential solution $g$. It turns out that, if we just choose $W$ small enough, then there exist $C > 0$, $\lambda \in \Q_{> 0}$ and $m \in \{1,\ldots,d\}$ such that 
\[
\operatorname{vol} T_p(\varepsilon) \approx C \cdot \varepsilon^\lambda (- \log \varepsilon)^{m-1},
\]
where $\log$ denotes the natural logarithm and the pair $(\lambda,m)$ is independent of the choice of $W$, ~\cite[Lemma 7.2]{Lin:tubes}. More precisely, there exists a choice for $W$ such that, for all open $W' \subseteq W$, the values of $\lambda$ and $m$ stay the same. The exact meaning of $\approx$ is explained in~\cite{Lin:tubes}. For us it is enough to understand that the two expressions are asymptotically equal, that is, their quotient goes to $1$ as $\varepsilon \to 0$.

The number $\lambda$ is called the \emph{real log canonical threshold} of $p$ at $a$ and $m$ is called its \emph{order}.
We write 
\[
\rlct_a(p) = (\lambda,m)
\]
and, by abuse of terminology, also call this pair the \emph{real log canonical threshold} of $p$ at $a$. The following result gives an alternative way to look at $\rlct_a(p)$. See~\cite{Atiyah:1970} for the original reference. The more recent paper~\cite[Corollary 3.10]{lin:phd} uses our terminology.

\begin{theorem}
    With the notations and conventions established above, we define the \emph{zeta function} $\R_{> 0} \to \R$ associated to $p$ around $a$ as
    \[
    \zeta_{p,a}(z) = \int_W \vert p(w)\vert^{-z}dw.
    \]
    It can be uniquely continued to a meromorphic function $\C \to \C$ all of whose poles are positive rational numbers. Its smallest pole is $\lambda$ and the order of this pole is $m$.
\end{theorem}

We will now make precise the connection between the value of the RLCT and the asymptotics of calculation errors.

\begin{remark}\label{remark:asymptotics}
    If we are given two potential solution paths $g_1$ and $g_2$, this leads to two error functions $p_1 = f-g_1$ and $p_2 = f- g_2$. Denote $\rlct_a(p_1) = (\lambda_1,m_1)$ and $\rlct_a(p_2) = (\lambda_2,m_2)$. Note that the following are equivalent for small enough $\eta > 0$:
\begin{enumerate}
    \item $\varepsilon^{\lambda_1} (- \log \varepsilon)^{m_1-1} > \varepsilon^{\lambda_2} (- \log \varepsilon)^{m_2-1}$ for all $\varepsilon \in (0,\eta)$.
    \item $\lambda_1 < \lambda_2$, or $\lambda_1 = \lambda_2$ and $m_1 > m_2$.
\end{enumerate}
Based on this observation, we define a total ordering on pairs
 $(\lambda_1,m_1),(\lambda_2,m_2) \in \R^2$ via
 \begin{equation}\label{eq:order}
 (\lambda_1,m_1) < (\lambda_2,m_2) \text{ if and only if } (\lambda < \lambda' \text{ or } (\lambda = \lambda' \text{ and } m > m')).
 \end{equation}
 See Example~\ref{example:easy-log-res} below for simple examples that realize several different values for $(\lambda,m)$. For general formulas in the case where the zeros of $f-g$ form a union of affine linear spaces, that is, where $f-g$ is a product of linear polynomials, see~\cite{KostaWindisch}.

  We can now easily compare two potential solutions by their RLCTs.
     Indeed, let $g_1$ and $g_2$ be two potential solutions for an exercise $(a,f)$. In this case, we could say that $g_1$ is \emph{asymptotically better} than $g_2$ if $\rlct_a(f-g_1) < \rlct_a(f-g_2)$ with respect to the total ordering defined in (\ref{eq:order}).
\end{remark}

Throughout the manuscript, if $f_1,\ldots,f_r \in \R[x_1,\ldots,x_d]$ are polynomials with real coefficients, we will denote by $V(f_1,\ldots,f_r) = \{x \in \R^d \mid f_1(x) = \cdots = f_d(x) = 0\}$ the real common zeros of the $f_i$. For instance, $V(f-g)$ denotes the set of all elements of $\R^d$ on which $f$ and $g$ agree. When convenient, one might think of $V(f_1,\ldots,f_r)$ as carrying some extra structure, e.g., the structure of a real analytic space or of an algebraic variety/scheme.
 \begin{remark}\label{remark:worst-case}
     It can be shown that $0 <  \lambda \leq d/2$ and $m \in \{1,\ldots,d\}$, where $\rlct_a(f-g) = (\lambda, m)$ for given $f,g \in \R[x_1,\ldots,x_d]$. Moreover, if $V(f-g) = \{x \in \R^d \mid f(x) - g(x) = 0\}$ has codimension $1$ in $\R^d$ then $\lambda \leq 1$. Consequently, in this last case, the worst asymptotic behavior of $f-g$ arises when $\rlct_a(f-g) = (1,1)$.
 \end{remark}

 \section{Resolution of singularities for asymptotic error estimation}\label{sec:resolution} With $\rlct_a(f-g)$, we now have a good measure on a theoretical level for the error that we make by replacing the solution path $f$ with a potential solution path $g$. However, there is no obvious straightforward way to compute this pair. Neither directly computing the asymptotics of $\operatorname{vol} T_{f-g}(\varepsilon)$ nor analyzing the poles of $\zeta_{f-g,a}$ seems feasible. Algebraic / analytic geometry gives a way to deal with this problem.

 The idea is to recursively apply analytic transformations to the domain of $p = f-g$ that will, in the end, transform $p$ to a monomial. Computing the RLCT of a monomial is then easy, as we will see. In order to make all of this notationally convenient, we introduce a variant of the RLCT.

 Let $\phi: W \to \R$ be an analytic function. We define 
 \[
\rlct_a(p;\phi) 
 \]
to be the pair consisting of the smallest pole and its order of the meromorphic function $\C \to \C$ given by
\[
\int_W |p(w)|^{-z} |\phi(w)|dw
\]
all of whose poles can be shown to be positive rational numbers, where $W$ is a small enough open neighborhood of $a$ with compact closure, see above. Note that $\rlct_a(p) = \rlct_a(p;1)$, where $1$ denotes the constant function taking the value $ 1\in \R$ on all of $W$. There is an equivalent formulation of the definition of $\rlct_a(p;\phi)$ via volumes, just as for $\rlct_a(p)$, see~\cite{Lin:tubes}. It amounts to replacing the Lebesgue measure $dw$ by the measure $|\phi(w)|dw$ and then computing the volume of $T_p(\varepsilon)$.

\subsection{Blow-ups and isomorphisms}
The transformations we use in order to achieve our goal are of two kinds: blow-ups and isomorphisms. The basic idea of blow-ups is to smoothen singularities by ``stretching out'' the variety along these singularities. The general construction is slightly involved, but for our purpose in this paper, blow-ups along linear subspaces of $\R^d$ are sufficient. We will describe these now. The \emph{blow-up} of $\R^d$ along the linear subspace $V \subseteq \R^d$ defined by $x_1 = \cdots = x_r = 0$ with $r \in \{1,\ldots,d\}$ is a map $\rho: M \to \R^d$, where $M$ is an algebraic subvariety of a product of $\R^d$ with real projective space of dimension $r-1$.
As will turn out, the exact nature of $M$ is not particularly relevant for us, as long as we can describe $\rho$ locally on $M$. Indeed, $M$ can be covered by $r$ open charts that are analytically isomorphic to $\R^d$ and the restrictions of $\rho$ to these charts are
    \begin{align}\label{equation:blow-up}
        \rho_i: \R^d &\longrightarrow \R^d \nonumber\\
                (x_1,\ldots,x_d) &\longmapsto (x_1x_i,\ldots,x_{i-1}x_i,x_i,x_{i+1}x_i,\ldots,x_rx_i,x_{r+1},\ldots,x_d).
    \end{align}
The Jacobian determinant of $\rho_i$ can be easily computed and is equal to $x_i^{r-1}$.

\begin{example}\label{example:blow-up}
    Let $\rho$ be the blow-up of $\R^2$ along the origin defined by $x = y = 0$. The first chart is just the map
    \begin{align*}
        \rho_1: \R^2 &\to \R^2\\
        (x,y) &\mapsto (x,xy).
    \end{align*}
    It is an isomorphism on $\R^2\setminus\{x = 0\}$ and maps exactly the line $x = 0$ to the origin. Similarly, the $i$-th chart $\rho_i$ in (\ref{equation:blow-up}) of the blow-up of $\R^d$ along the subspace $x_1 = \ldots = x_r = 0$ pulls this subspace back to the hyperplane $x_i = 0$ and is an isomorphism away from this hyperplane.

\begin{figure}[H]
\centering
\begin{tikzpicture}
\node (A) at (0,0)
    {\includegraphics[width=0.35\textwidth]{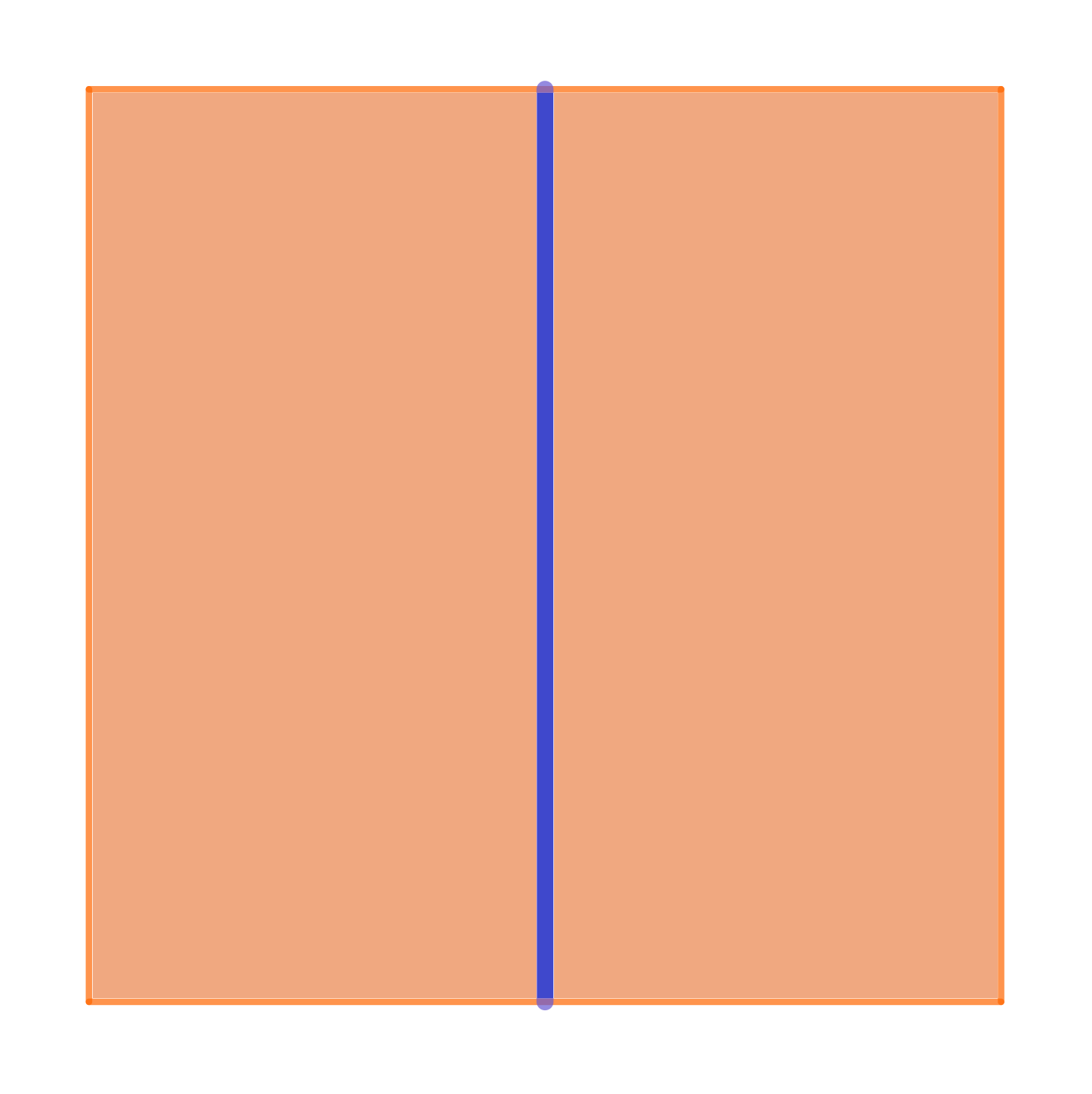}};
\node (B) at (8,0)
    {\includegraphics[width=0.35\textwidth]{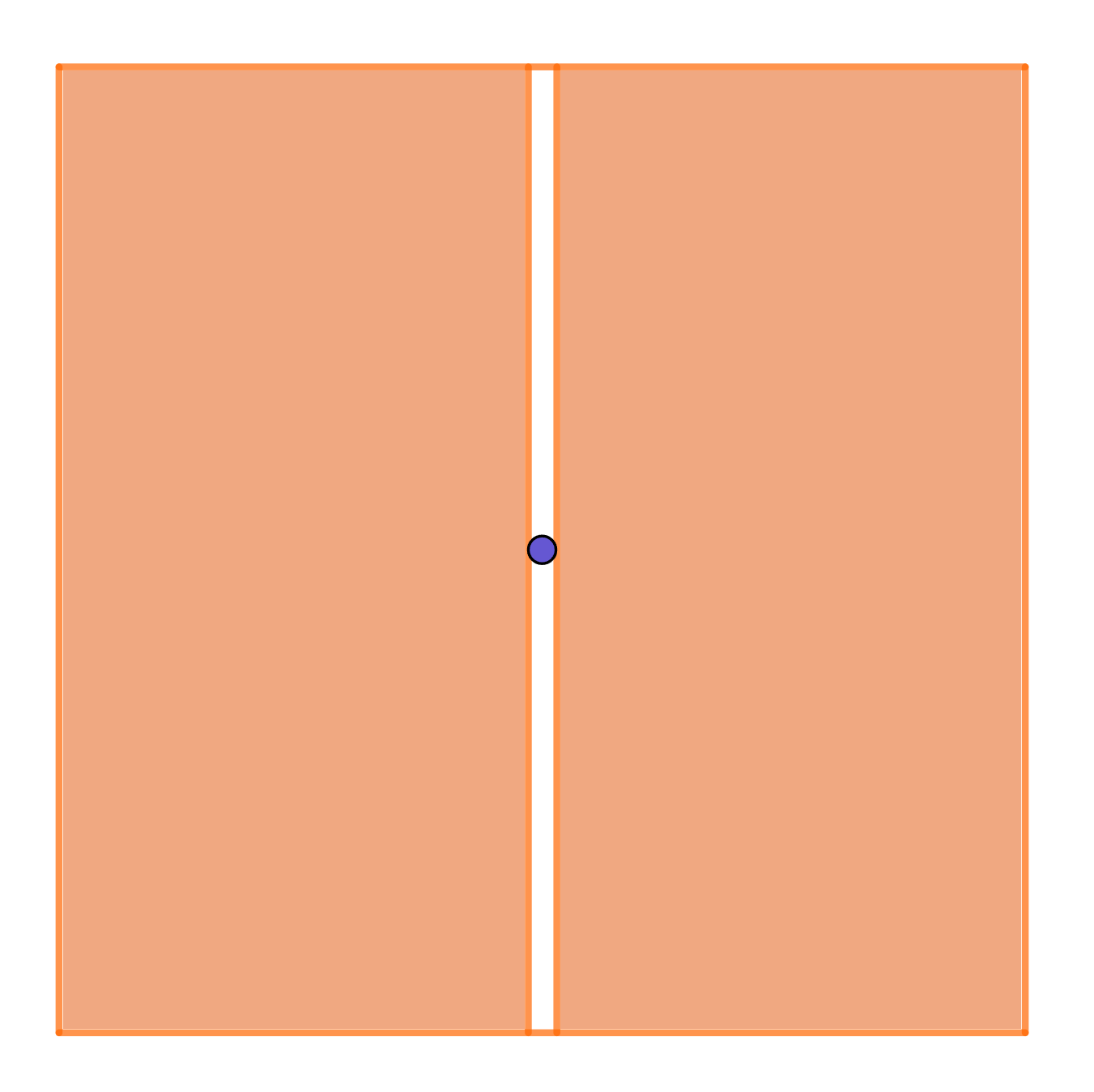}};

\draw[->, thick]
    (A.east) -- node[above] {$\rho_1$} (B.west);
\end{tikzpicture}
\caption{The first chart $\rho_1$ of the blow-up of $\R^2$ along the origin pulls this point back to the line $\{x=0\}$ and is an isomorphism away from this line.}
\label{fig:rho1}
\end{figure}
\end{example}

Now that we have blow-ups at hand, we can review a few useful facts from the literature on which our recursive method to compute $\rlct_a(p)$ will be based. The main idea is, that blow-ups will transform the function $p$ to a monomial and thereby not change the RLCT.

\begin{remark}\label{remark:facts}
    \begin{enumerate}
        \item \cite[Lemma 3.8]{lin:phd} Let $\chi: W \to \R$ be analytic and non-vanishing on all of $W$. Then \[\rlct_a(p;\chi\cdot \phi) = \rlct_a(p;\phi).\]
        \item \cite[Proposition 8]{Lin2017Ideal} Let $\rho_1,\ldots,\rho_r$ be the charts of a blow-up as described above. Then
        \[
        \rlct_a(p;\phi) = \min_{\substack{i \in [r] \\ y \in \rho_i^{-1}(a)}} \rlct_y(p \circ \rho_i; (\phi \circ \rho_i)\cdot x_i^{r-1}).
        \]
        \item \cite[Proposition 8]{Lin2017Ideal} Let $W' \subseteq \R^d$ and $\alpha: W' \to W$ be a real analytic isomorphism, that is, a bijective analytic map whose inverse is also analytic. Then
        \[
        \rlct_a(p;\phi) = \rlct_{\alpha^{-1}(a)}(p\circ \alpha; \phi \circ \alpha).
        \]
        \item The following is a combination of \cite[Theorem 7.1]{Lin:tubes} and \cite[Proposition~6]{Lin2017Ideal}. Let $a_1,\ldots,a_d,b_1,\ldots,b_d \in \Z_{\geq 0}$ and let $\mu$ be an analytic function that does not vanish on $W$. If we write
        $
        \rlct_0(\mu \cdot x_1^{a_1} \cdots x_d^{a_d}; x_1^{b_1} \cdots x_d^{b_d}) = (\lambda,m)
        $
        then
         \begin{align*}
             \lambda &= \min_{\substack{i \in [d] \\ a_i \neq 0 }} \frac{b_i +1}{a_i}, \\
             m &= \bigg\vert  \left\{ i \in [d] \ \bigg\vert \ a_i \neq 0,  \lambda = \frac{b_i +1}{a_i} \right\} \bigg\vert.
         \end{align*}
         Moreover, for all $a \in \R^d$, 
        \[
         \rlct_0(\mu \cdot x_1^{a_1} \cdots x_d^{a_d}; x_1^{b_d} \cdots x_d^{b_d}) \leq \rlct_a(\mu \cdot x_1^{a_1} \cdots x_d^{a_d}; x_1^{b_1} \cdots x_d^{b_d})
        \]
    \end{enumerate}
\end{remark}

Now that we have Remark~\ref{remark:facts}(1)-(4) at hand, we can make more precise what we meant by the recursive method for computing $\rlct_a(p;\phi)$. We start with $p$ that is defined on $W \subseteq \R^d$. We then apply an isomorphism $\alpha$ as in (3) that depends on $p$. This does not change the $\rlct$. We then do a blow-up along a well-chosen linear subspace which gives us chart maps $\rho_1,\ldots,\rho_r$. Now, by (2), renaming $p_i = p\circ\alpha \circ \rho_i$ and $\phi_i = (\alpha \circ \rho_i) \cdot x_i^{r-1}$, we are left with computing $\rlct_y(p_i; \phi_i) $ for all $i \in [r]$ and all $y \in (\alpha \circ \rho_i)^{-1}(a)$. We now can apply the same process to this new problem with $p = p_i$, $\phi = \phi_i$ and $a = y$ and we hope that, after finitely many steps, we arrive at a situation where both $p$ and $\phi$ are monomials as in (4). In Section~\ref{sec:helpful-results}, we will introduce further useful results on RLCTs and we will then illustrate in several instructive examples how the above process works.

The composition of all the isomorphisms and blow-ups that we use in order to make $p$ and $\phi$ monomials is called a \emph{log resolution of singularities} in the literature. As stated above, there are more general blow-up constructions than just the one along a \emph{center} that is a linear subspace. A famous result of Hironaka~\cite{Hironaka1964Resolution} on resolution of singularities states that, if we allow these general blow-ups, then a log resolution always exists. 

For some readers, it might be confusing that we call a map, that transforms $p$ locally to a monomial, a resolution of singularities. Indeed, what we usually mean by a resolution of singularities is a smooth variety that is isomorphic to the one defined by $p = 0$ on a large subset.\footnote{The notion of ``large'' depends on the context. It means usually either being the complement of a vanishing set of polynomials or of analytic functions.} It is easy to find such a resolution of singularities from a log resolution $f$ by removing from $\{p\circ f = 0\}$ the set on which $f$ is not bijective and then taking the closure. The converse is not true: For instance, in Example~\ref{example:easy-log-res}(4) below, a single blow-up yields a resolution of singularities for $p = x^2 - y^3 = 0$ on $\R^2$, but only a sequence of three blow-ups leads to a log resolution.

\subsection{Examples of RLCTs}\label{sec:helpful-results}

Before we calculate the asymptotics of $\varepsilon$-tubes around special varieties in $\R^2$, we review two useful results on isomorphic transformations of certain unions of hypersurfaces. Let $p_1,\ldots,p_r \in \R[x_1,\ldots,x_d]$, $a \in \R^d$, and suppose that each hypersurface $V(p_i)$ contains $a$ and is non-singular at $a$. We say that $V(p_1),\ldots,V(p_r)$ \textit{meet transversally} at $a$ if the intersection $T_a V(p_1) \cap \cdots \cap T_a V(p_r)$ of tangent spaces at this point has codimension $r$ in $\R^d$, see Figure~\ref{fig:trans} for illustration. In this case, we also say that $p = p_1\cdots p_r$ defines a divisor with \emph{simple normal crossings} on $\R^d$. Of course, if this is true, $r$ can be at most $d$. The following two lemmas are common knowledge but, as their proofs are easy and strengthen intuition, we summarize them here.

\begin{figure}[H]
    \centering
         \includegraphics[width=0.45\textwidth]{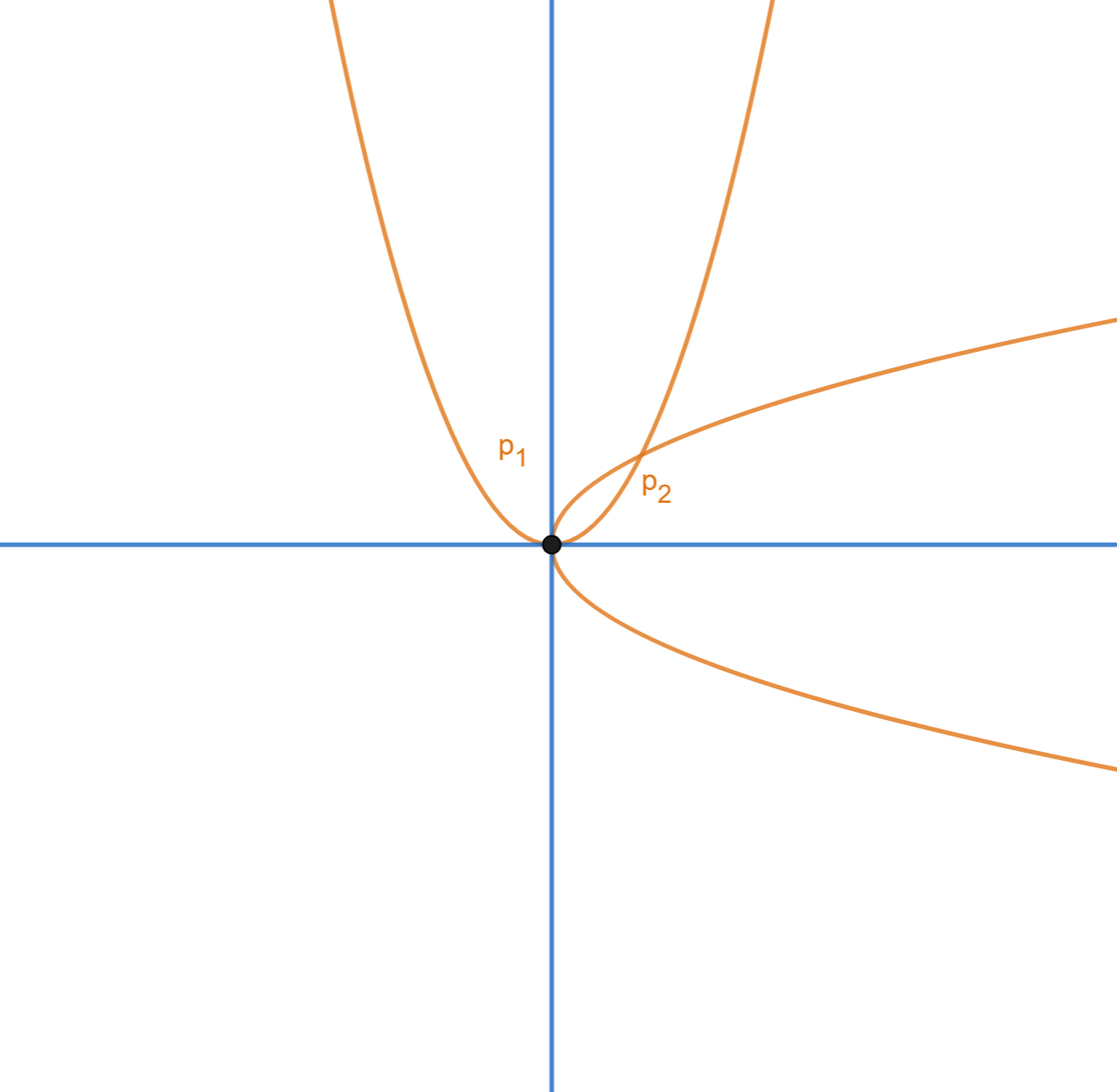} \includegraphics[width=0.45\textwidth]{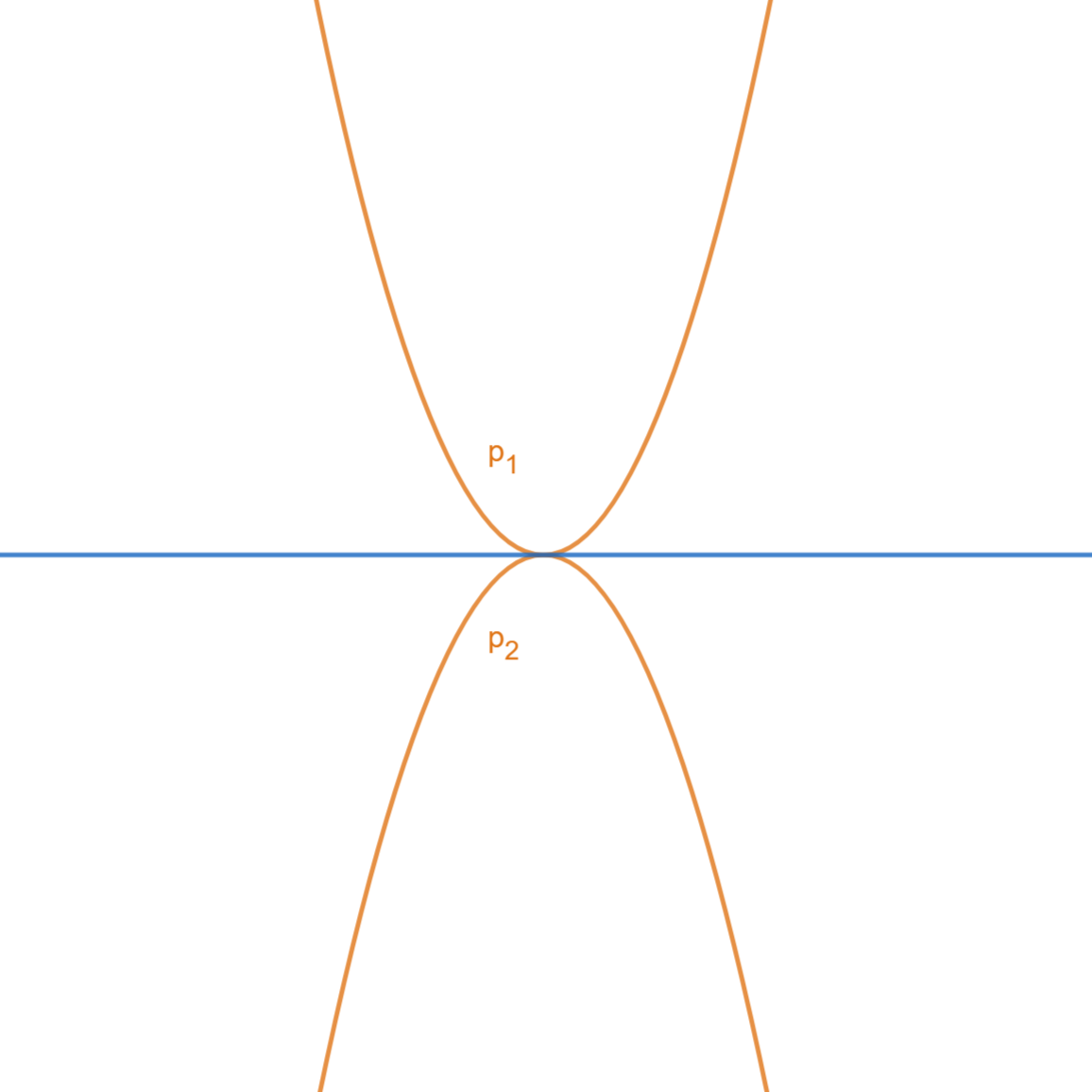}
    \caption{On the left, the two parabolas (orange) defined by $p_1 = y - x^2$ and $p_2 = y^2 - x$ meet transversally at the origin as their tangent spaces (blue) intersect in a point (black), which has codimension $2$. On the right, the two parabolas (orange) defined by $p_1 = y-x^2$ and $p_2 = y + x^2$ do not meet transversally at the origin as they have the same tangent space (blue).}
    \label{fig:trans}
\end{figure}

\begin{lemma}\label{lemma:transversal}
    Let $p_1,\ldots,p_r \in \R[x_1,\ldots,x_d]$, $a \in \R^d$, and suppose that each hypersurface $V(p_i)$ contains $a$, is non-singular at $a$, and that they meet transversally at $a$. Set $p = p_1^{e_1}\cdots p_r^{e_r}$ with $e_i \in \Z_{\geq 0}$. Then there exist  neighborhoods $U,V \subseteq \R^d$ of $a$ and $0$, respectively, and a real analytic isomorphism $\sigma: V \to U$ with $\sigma(0) = a$ and $p \circ \sigma = x_1^{e_1}\cdots x_r^{e_r}$, where the $x_i$ are the first $r$ coordinate functions on $\R^d$.
\end{lemma}

\begin{proof}
    Consider $F = (p_1,\ldots,p_r): \R^d \to \R^r$. Since the $p_i$ are non-singular and meet transversally at $a$, the rank of the Jacobian matrix $\jac F$ at $a$, which equals the codimension of $T_a V(p_1) \cap \cdots \cap T_a V(p_r)$, is exactly $r$. Complete the rows of $\jac F$ at $a$ to a basis of $\R^d$ by adding vectors $c_{r+1},\ldots, c_d \in \R^d$ and define linear polynomials $\ell_i = \sum_{j = 1}^d c_{ij} (x_j - a_j)$ for $i \in \{r+1,\ldots,d\}$. Now, the map $\widehat{F} = (p_1,\ldots,p_r,\ell_{r+1},\ldots,\ell_d): \R^d \to \R^d$ has Jacobian matrix of full rank at $a$. Consequently, the real analytic inverse function theorem~\cite[Theorem 1.8.1]{KrantzParks2002} says that there are neighborhoods $U,V \subseteq \R^d$ of $a$ and $0$, respectively, such that $\widehat{F}: U \to V$ has an analytic inverse $\sigma: V \to U$. It follows that $\sigma(0) = a$ as $\widehat{F}(a) = 0$. Moreover, $\widehat{F} \circ \sigma = \operatorname{id}_V = (x_1,\ldots,x_d)$ and the $i$-th coordinate of $\widehat{F}$ is $p_i$, so $p_i \circ \sigma = x_i$.
\end{proof}

\begin{lemma}\label{lemma:regular}
    Let $p \in \R[x_1,\ldots,x_d]$ define a hypersurface in $\R^d$ that is non-singular at $a \in \R^d$. Then $\rlct_a(p) = (1,1)$.
\end{lemma}

\begin{proof}
    By Lemma~\ref{lemma:transversal}, there exist a neighborhood $U \subseteq \R^d$ of $a$ and a real analytic isomorphism $\sigma: \R^d \to U$ with $\sigma(0) = a$ and $p \circ \sigma = x_1$. The result now follows immediately from Remark~\ref{remark:facts}(4).
\end{proof}

\begin{example}\label{example:easy-log-res}
    We illustrate the recursive process of applying isomorphisms and blow-ups in several easy examples.
    \begin{enumerate}
        \item The first example is just the line in $\R^2$ defined by the vanishing of the function $p(x,y) = y$. Here, we do not even need to start the recursive progress described above. Instead, we can directly use Remark~\ref{remark:facts}(4) on $\rlct_a(y)$ which gives $(1,1)$ for every $a$ on this line.
         \begin{figure}[H]
    \centering
    \includegraphics[width=0.5\textwidth]{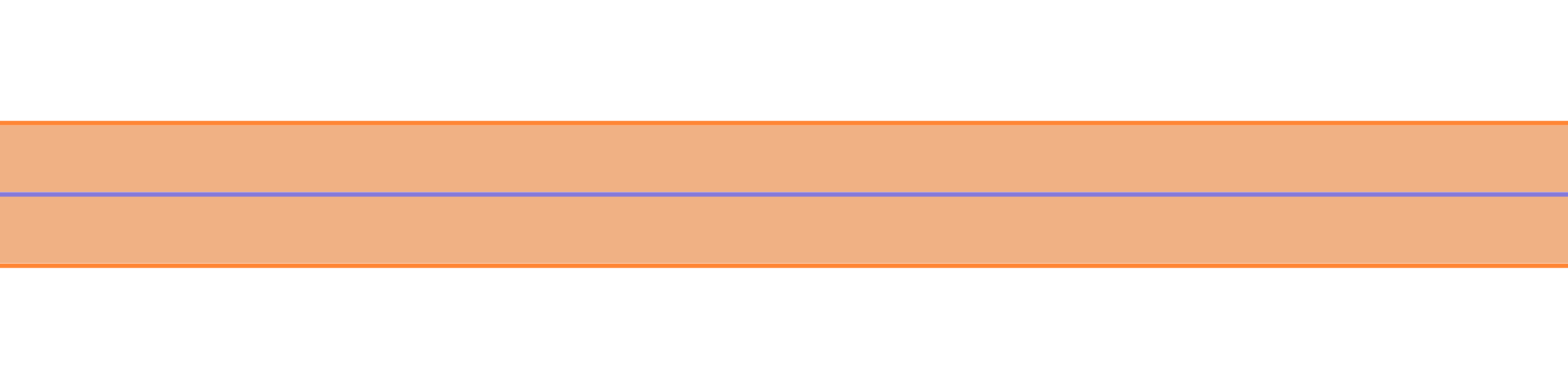}
    \caption{The tube $T_p(\varepsilon)$ (orange) around the variety given by $p(x,y)= y = 0$ (purple). As $\varepsilon \to 0$ the area of the tube is asymptotically $C\cdot \varepsilon$.}
    \label{fig:line}
\end{figure}
        \item The ``freshman's dream'' is the wrong believe that $(x+y)^2$ would equal $x^2 + y^2$ for all pairs of real numbers. The corresponding polynomial of differences is $p(x,y) = (x+y)^2 - (x^2 +y^2) = 2xy$. Again, up to a non-vanishing (or even constant) function, this is a product of variables. So we can compute $\rlct_0(2xy) = (1,2)$ using Remark~\ref{remark:facts}(4). We will study this in detail for a growing number of variables in Section~\ref{sec:freshman}.
        \begin{figure}[H]
    \centering
    \includegraphics[width=0.5\textwidth]{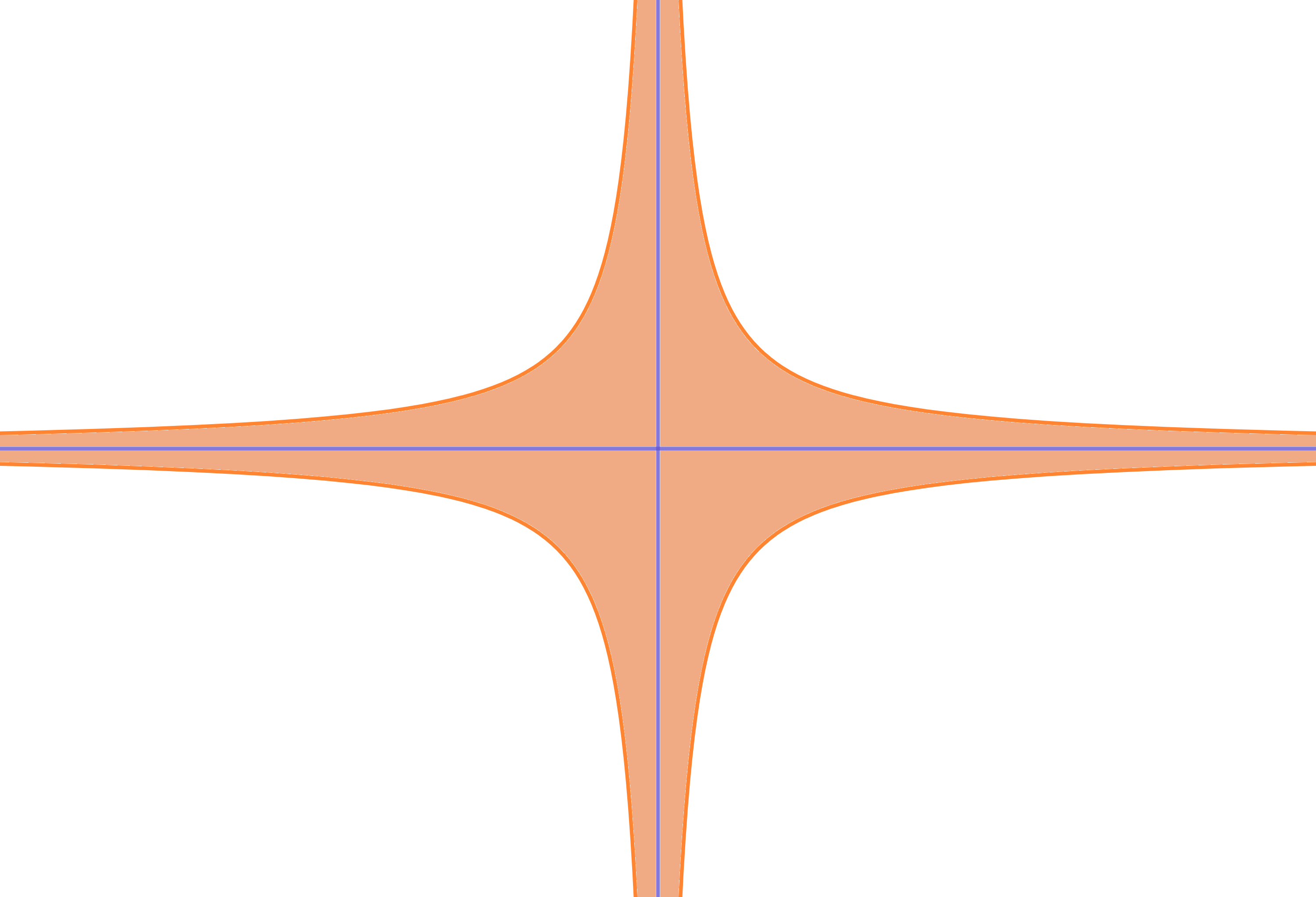}
    \caption{The tube $T_p(\varepsilon)$ (orange) around the variety given by $p(x,y)=xy = 0$ (purple). As $\varepsilon \to 0$ the area of the tube is asymptotically $C\cdot \varepsilon(-\log \varepsilon)$.}
    \label{fig:xy}
\end{figure}

        \item The variation of the ``freshman's dream'' complementary to Section~\ref{sec:freshman} is the one where, instead of the number of variables, the exponent is raised.  That is, one compares $(x + y)^3$ against $x^3 + y^3$ whose difference is, after division by $3$, equal to $p(x,y) = xy(x+y)$. This polynomial has three irreducible factors whose corresponding hypersurfaces meet in the origin of $\R^2$. So they do not meet transversally at this point. We will therefore need to start our recursive method of blowing up to study $\rlct_0(p)$.
             \begin{figure}[H]
    \centering
         \includegraphics[width=0.6\textwidth]{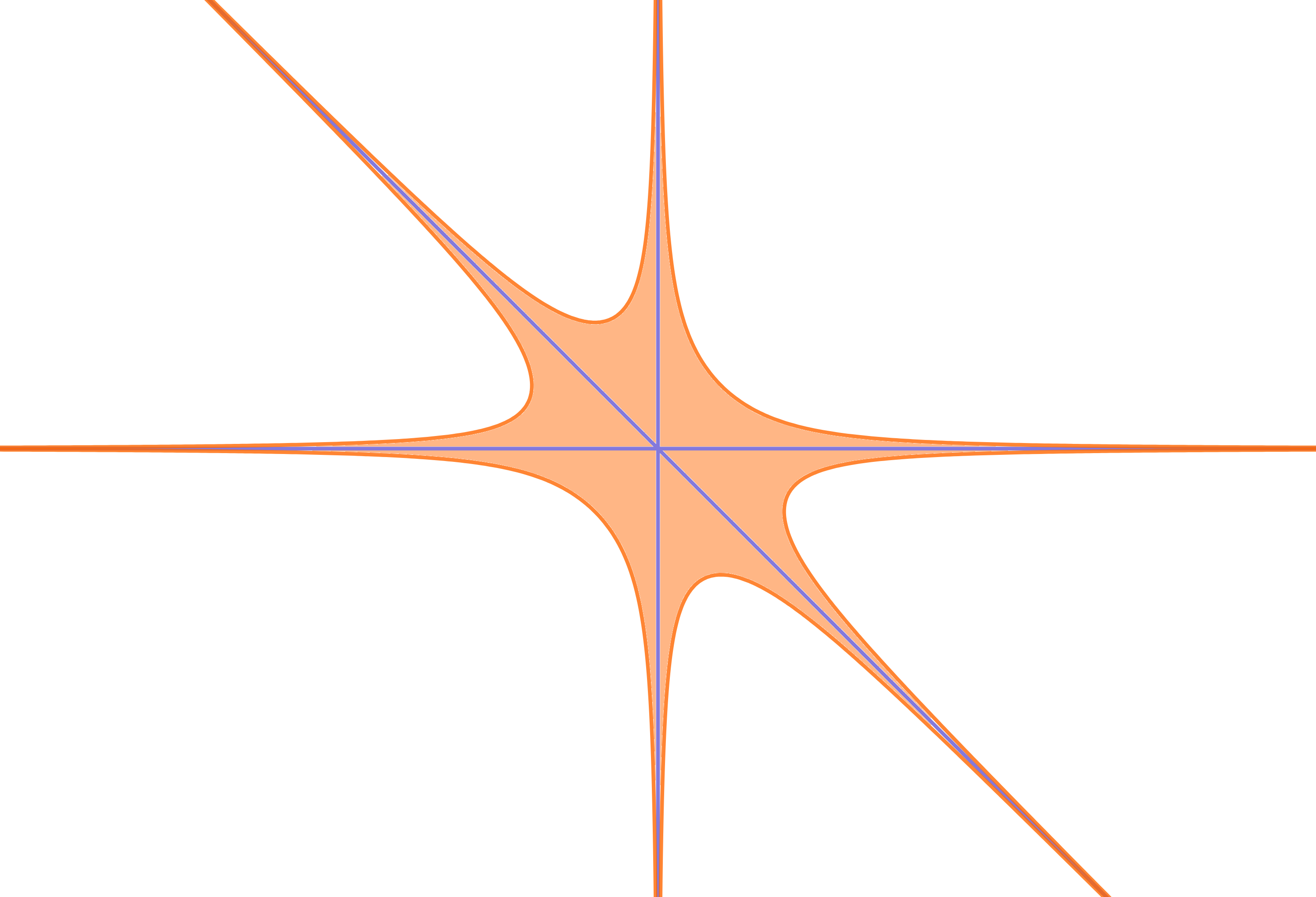}
    \caption{The tube $T_p(\varepsilon)$ (orange) around the variety given by $p(x,y)= xy(x+y) = 0$ (purple). As $\varepsilon \to 0$ the area of the tube is asymptotically $C\cdot \varepsilon^{2/3}$.}
    \label{fig:xy(x+y)}
\end{figure}

    The first isomorphism $\sigma$ is just the identity. Next, we apply the blow-up of $\R^2$ at the origin. Since $p$ is symmetric in $x$ and $y$, we just have to consider one chart, say $\rho_1(x,y) = (x,xy)$, which gives $p \circ \rho_1 = x^3y(1+y)$ and has Jacobian determinant $x$. A point $(x^*,y^*)$ lies in $\rho_1^{-1}(0)$ if and only if $x^* = 0$. In small neighborhoods of points $(0,y^*)$ with $y^* \notin \{0,-1\}$, the two factors $y$ and $1+y$ do not vanish. Hence, Remark~\ref{remark:facts}(4) gives that the first component of $\rlct_{(0,y^*)}(x^3y(1+y);x)$ is
    $(1+1)/3 = 2/3$.
    This value just appears once as a minimum among the factors, so the second component of the RLCT is $1$. The two points $(0,-1)$ and $(0,0)$ behave the same up to a linear isomorphism. So, we only consider $\rlct_0(x^3y(1+y);x)$ whose first component is computed by
    \[
    \min \left\{ \frac{1+1}{3}, \frac{0+1}{1}  \right\} = \frac{2}{3}.
    \]
    So, minimizing over all points in $\rho_1^{-1}(0)$ by using Remark~\ref{remark:facts}(2), we get \[\rlct_0(xy(x+y))= (2/3,1).\]

        \item We consider the cusp $p(x,y) = x^2 - y^3$ and will prove that $\rlct_0(p) = (5/6,1)$. A first blow-up will give a resolution of singularities in the sense that it transforms the variety given by $x^2 - y^3 = 0$ birationally to a non-singular variety. But only after two further blow-ups, we will reach a log resolution. We start with the first blow-up at the origin. \\
        \textbf{$x$-chart.} We compute $p \circ \rho_1 = x^2(1 - xy^3)$ with Jacobian determinant of the blow-up chart equal to $x$. The second factor does not vanish on the preimage of the origin. Therefore, we can disregard it by Remark~\ref{remark:facts}(4) and just compute $\rlct_0(x^2;x) = (1,1)$.\\
        \textbf{$y$-chart.} Here, we see that $p \circ \rho_2 = y^2(x^2 - y)$
        whose zero locus is the parabola given by $y = x^2$ together with the line tangent to it in $0$. So, these do not meet transversally and we cannot use Lemma~\ref{lemma:transversal} together with Remark~\ref{remark:facts}(4). Another blow-up will help us out. We have to take into account the Jacobian determinant $y$ from the first blow-up. So, we have the new function $\tilde{p}(x,y) = y^2(x^2 - y)$ and apply the blow-up at the origin. 

    The $y$-chart gives $y^3(x^2y - 1)$ with total Jacobian determinant $y^2$. The second factor does not vanish at points in the preimage of the origin under the two blow-up maps, so we compute $\rlct_0(y^3;y^2) = (1,1)$.

    The $x$-chart gives $x^3y^2(x-y)$ with total Jacobian determinant $x^2y$. This is a union of three lines meeting at the origin and, hence, still not isomorphic to a monomial. A further blow-up gives $x$-chart $x^6y^2(1-y)$ with Jacobian determinant $x^4y$, which has minimal RLCT $(5/6,1)$ at the origin. The $y$-chart is $x^3y^6(x-1)$ with Jacobian determinant $x^2y^4$. This has RLCT $(5/6,1)$.
    So the minimum over all charts arises with $(5/6,1)$ as claimed. Note that, with this, the RLCT can distinguish plane curve singularities, compare the node in (2) above.
    
        \begin{figure}[H]
    \centering
    \includegraphics[width=0.8\textwidth]{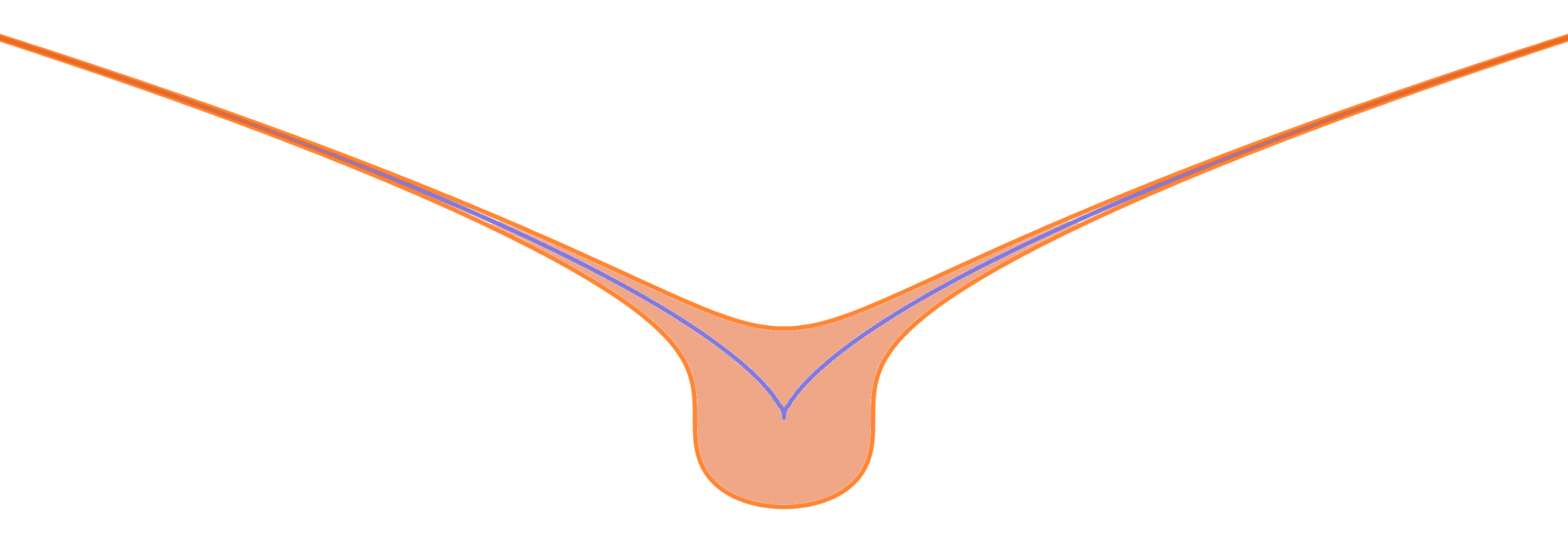}
    \caption{The tube $T_p(\varepsilon)$ (orange) around the variety given by $p(x,y)= x^2 - y^3 = 0$ (purple). As $\varepsilon \to 0$ the area of the tube is asymptotically $C\cdot \varepsilon^{5/6}$.}
    \label{fig:cusp}
\end{figure}

\item In a last example, we want to illustrate two things: that the base field for the log resolution matters and that, over $\R$, the first entry of the RLCT can get as large as $d/2$ for a polynomial $p \in \R[x_1,\ldots,x_d]$. We start with a low-dimensional example.

Let $p(x,y) = x^2 + y^2$. Over $\C$ this factors as $(x+iy)(x-iy)$ and so, up to a linear isomorphism, is the same as $xy$ from (2) which has RLCT $(1,2)$. Over $\R$, the polynomial is irreducible and its zero set consists just of one point, the origin. So it has codimension $2$ although over $\C$, the polynomial defines the union of two lines and has, hence, codimension $1$. A blow-up at the origin gives two symmetric charts with pullback of $p$ given by $x^2(1+y^2)$ and Jacobian determinant $x$. Since $1+y^2$ does not vanish over $\R$, this gives $\rlct_0(p) = \rlct_0(x^2;x) = (1,1)$.

Now, we go to a higher number of variables. Consider $p = x_1^2 + \cdots + x_d^2$. The blow-up at the origin gives the pullback $x_1^2(1 + x_2^2 + \cdots + x_d^2)$ with Jacobian determinant $x_1^{d-1}$. Again the second factor of the pullback does not vanish over $\R$, so we compute $\rlct_0(p) = \rlct_0(x_1^2;x_1^{d-1}) = (d/2,1)$.
    \end{enumerate}
\end{example}

\section{Log resolutions for the freshman's dream variety}\label{sec:freshman}

The topic of this section is the asymptotic estimation of the error that is made by supposing that $(x_1 + \cdots + x_n)^2$ is equal to $x_1^2 + \cdots + x_n^2$. We show that the error is, except for the case $n = 2$, as high as possible, compare Remark~\ref{remark:worst-case}.

\begin{theorem}\label{theorem:freshman}
    Let $n \geq 2$ and 
    \[
    f_n = (x_1 + \cdots + x_n)^2 - (x_1^2 + \cdots + x_n^2) = \sum_{\substack{i,j = 1 \\ i < j}}^n 2x_ix_j.
    \]
    Moreover, let $a \in \R^n$ such that $f_n(a) = 0$.
    Then 
    \[
    \rlct_a(f_n)=
    \begin{cases}
        (1,2) \text{ if } a = 0 \text{ and } n = 2,\\
        (1,1) \text{ else.}
    \end{cases}
    \]
\end{theorem}

\begin{figure}[H]
    \centering
         \includegraphics[width=0.45\textwidth]{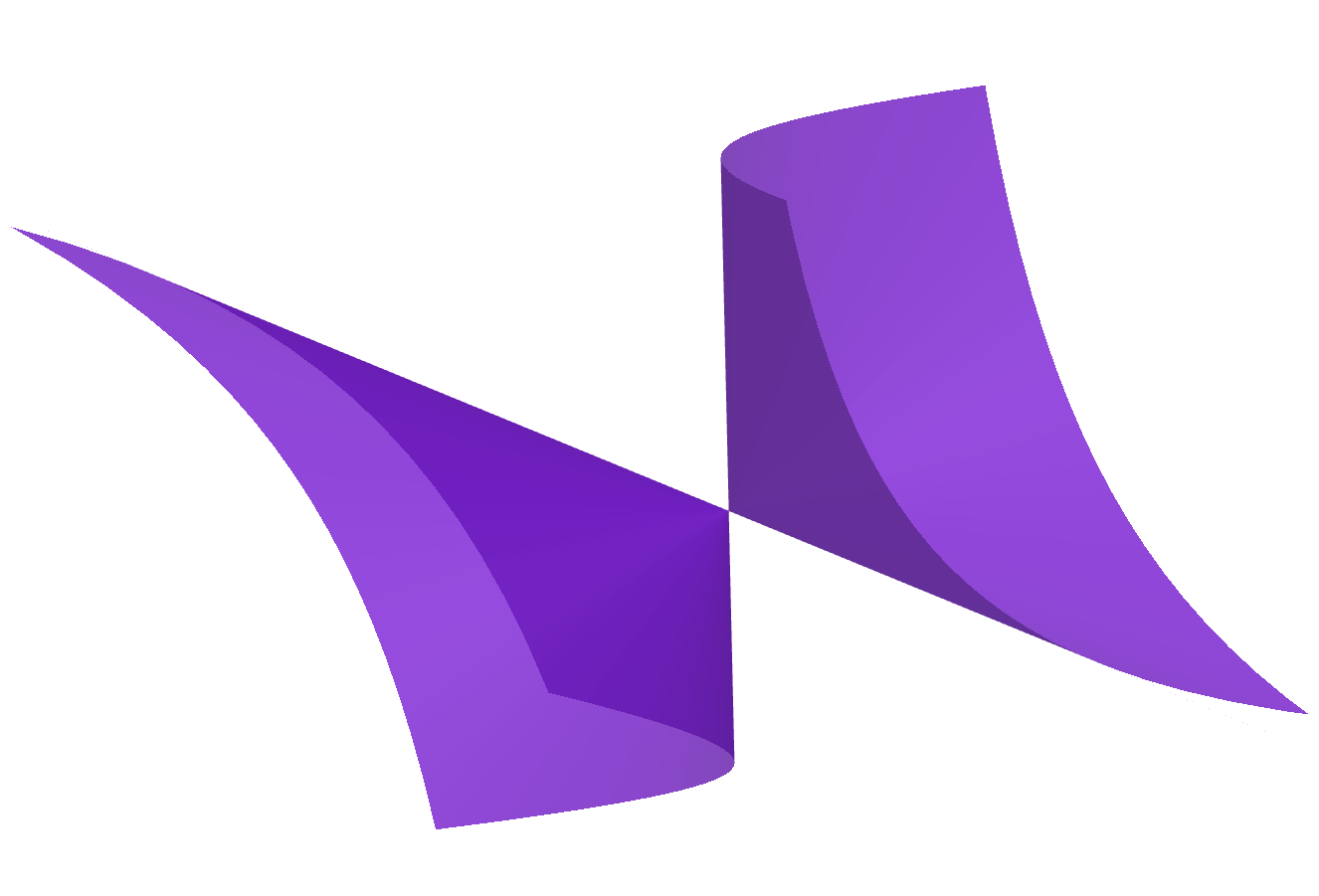} \includegraphics[width=0.45\textwidth]{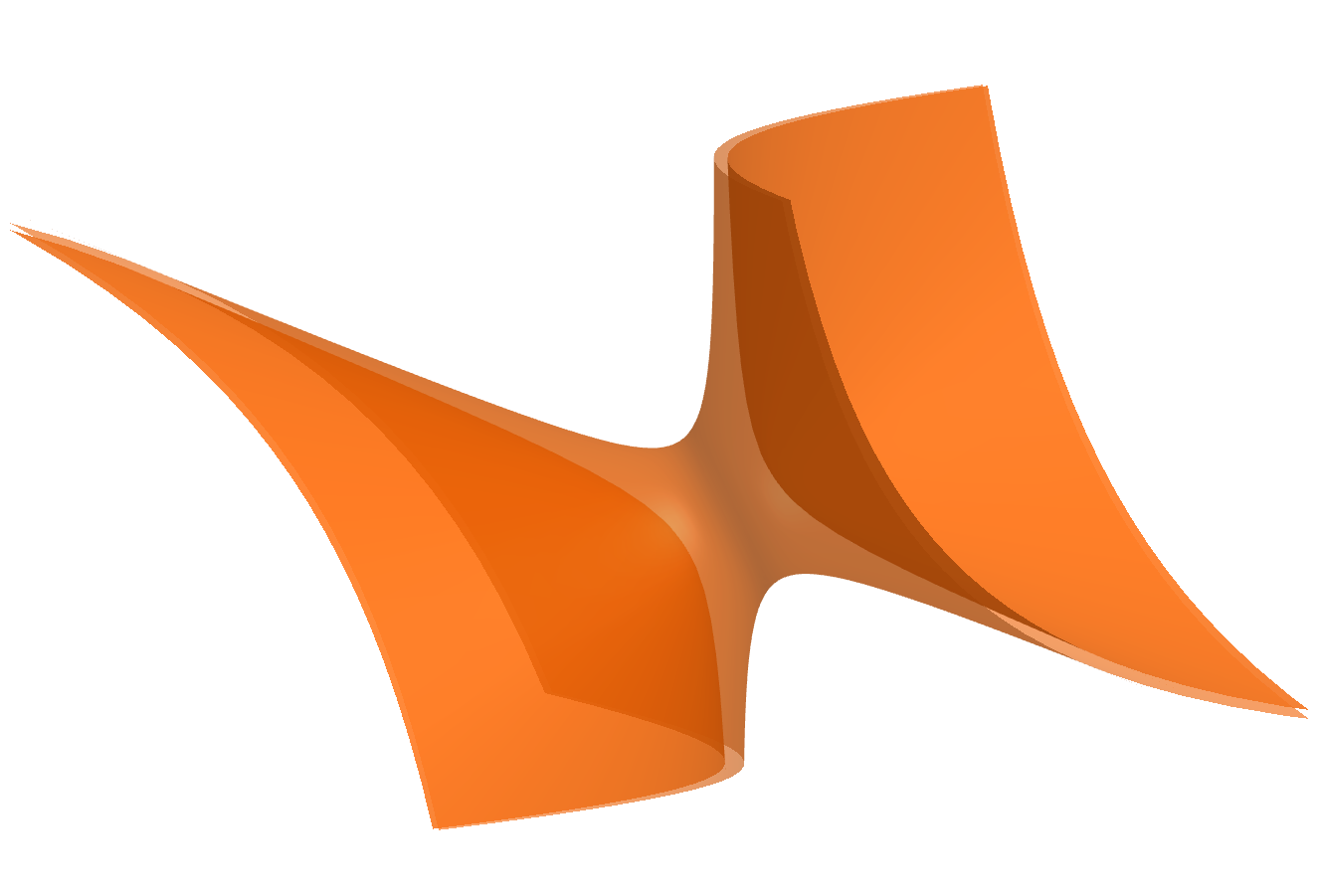}
    \caption{The cone (left) is the freshman's dream variety in the case $n = 3$ given by $f_3 = xy + yz + xz = 0$ with its $\varepsilon$-tube (right). Theorem~\ref{theorem:freshman} shows that the volume of the tube goes to $0$ like $C \cdot \varepsilon$ which is the worst possible behavior for a 2-dimensional space of points satisfying $f_3(x,y,z) = 0$ in $\R^3$.}
    \label{fig:cone}
\end{figure}

\begin{proof}
    We first investigate in which real points the variety defined by $f_n$ is non-singular. For $i \in \{1,\ldots,n\}$, the $i$-th entry (up to a scaling factor of $2$) of the Jacobian matrix of $f_n$, that is, of its gradient is
    \[
    J_i = \sum_{\substack{j = 1\\ j \neq i}}^n x_j.
    \]
    The singularities of the variety defined by $f_n$ are exactly the common zeros of $f_n$ with the $J_i$. The common zeros of the $J_i$ are exactly the points in the kernel of the matrix
    \begin{equation}\label{equation:matrix}
    \begin{bmatrix}
        0 & 1 & 1 & \cdots & 1 \\
        1 & 0 & 1 & \cdots & 1 \\
        1 & 1 & 0 & \cdots & 1 \\
        \vdots &  &  & \ddots & \\
        1 & 1 & 1 & \cdots & 0
    \end{bmatrix}
    \end{equation}
    which has full rank. So the only possible singular point is $a = 0 \in \R^n$ which also satisfies $f_n(a) = 0$ and is, hence, a singular point. Consequently, $\rlct_a(f_n) = (1,1)$ for $a \in \R^n\setminus \{0\}$ with $f_n(a) = 0$, see Lemma~\ref{lemma:regular}.

    We now turn to the case where $a = 0$. Since we saw that this is the only singularity of $V(f_n)$, we perform the blow-up at the origin, that is, the blow-up along the linear subspace defined by $x_1 = \ldots = x_n = 0$ and see if it yields a log resolution. Note that $f_n$ is invariant under any permutation of the variables and, hence, we need only consider one chart, say the chart $\rho = \rho_1$ associated to the variable $x_1$ which has Jacobian determinant $x_1^{n-1}$. We write
    \[
    f_n = \sum_{j = 2}^n 2x_1x_j + \sum_{\substack{i,j = 2 \\ i<j}}^n 2x_ix_j,
    \]
    giving
    \[
    f_n \circ \rho = 2x_1^2\cdot \underbrace{\left(  \sum_{j = 2}^n x_j + \sum_{\substack{i,j = 2 \\ i <j}}^n x_ix_j\right)}_{g_n(x_1,x_2,\ldots,x_n) :=}.
    \]
    We will now show that the variety associated to $g_n$ is smooth and meets $\{x_1 = 0\}$ transversally which allows us to use Lemma~\ref{lemma:transversal} to treat $g_n$ as a variable and compute the RLCT via Remark~\ref{remark:facts}(4). 

    The Jacobian matrix of $g_n$ is just the gradient
   \begin{equation}\label{equ:gradient}
    \left[ \ \ 0 \ \ \  1 + \sum_{\substack{j = 2 \\ j \neq 2}}^n x_j \ \ \ \cdots \ \ \ 1 + \sum_{\substack{j = 2 \\ j \neq n}}^n x_j \ \ \right].
    \end{equation}
    The vanishing of this vector, again, is a (inhomogeneous) linear system with coefficient matrix as in Equation~(\ref{equation:matrix}), now of dimension $n-1$, so has a unique solution (up to an arbitrary choice of $x_1$) which is $x_j = \frac{-1}{n-2}$ for $j \in \{2,\ldots,n\}$ in the case $n \geq 3$ and has no solution for $n = 2$. For $n \geq 3$, a simple calculation gives
    \[
    g_n\left(x_1, \frac{-1}{n-2},\ldots,\frac{-1}{n-2}\right) = \frac{1-n}{2(n-2)} 
    \]
    and, hence, this point never lies on the variety $V(g_n)$. Consequently, $V(g_n)$ is non-singular for all $n \geq 2$. 

    The tangent spaces at each point $(x_1,\ldots,x_n)$ of $V(x_1)$ and $V(g_n)$, respectively, are two hyperplanes with covectors $[1 \ 0 \ \cdots \ 0]$ and the vector in (\ref{equ:gradient}), respectively. These are non-zero vectors that are orthogonal and, in particular, not colinear. Therefore, there is no containment of the tangent spaces and, hence, the varieties meet transversally.

    By Lemma~\ref{lemma:transversal}, for each point $b \in V(x_1) \cap V(g_n)$, there exists a real analytic isomorphism $\sigma: U_0 \to U_b$, where $U_0$ and $U_b$ are small neighborhoods of $0$ and $b$, respectively, such that $\sigma(0) = b$ and $f_n \circ \rho \circ \sigma = 2x_1^2x_2$ and $\det \jac \rho \circ \sigma = x_1^{n-1}$. We now use Remark~\ref{remark:facts}(4) to compute $\rlct_0(2x_1^2x_2;  x_1^{n-1}) = (\lambda,m)$ which will be the desired pair by Remark~\ref{remark:facts}(2).

    In the case $n = 2$, we deduce for $\rlct_0(x_1^2x_2;x_1) = (\lambda,m)$ that
    \[
     \lambda = \min \left\{ \frac{1+1}{2}, \frac{0+1}{1} \right\} = 1
    \]
    and this minimum is attained twice, so $m = 2$. If $n \neq 2$ then we compute
    \[
    \lambda = \min \left\{ \frac{n-1+1}{2}, \frac{0+1}{1} \right\} = 1
    \]
    and this minimum is only attained once, so $m = 1$.
\end{proof}

\section{Singularity types of the petrol variety}\label{sec:petrol}

In this section, we explicitly compute the singular locus and the RLCT for the petrol variety $V(p_n)$ from Example~\ref{example:fuel-consumption} in several cases feasible by symbolic computations in the computer algebra system \emph{Macaulay2} \cite{M2}. Although the case $n = 2$ differs geometrically from $n\geq 3$, see below, the singular locus within the space \[(\R^*)^{2n} = \{(x_1,\ldots,x_n,y_1,\ldots,y_n) \in \R^{2n} \mid \forall i \ x_i \neq 0 \neq y_i\}\] is the same and given by 
\[
\Delta_n = \{(x_1,\ldots,x_n,y_1,\ldots,y_n) \in (\R^*)^{2n} \mid  x_1 = \cdots = x_n, y_1 = \cdots = y_n\}
\]
in each case, which we will call the \emph{diagonal} in the following.
%\footnote{The reader more familiar with algebraic geometry might be worried that we talk about varieties instead of schemes here. However, as the computations below will show, both the variety $V(p_n)$ and its singular locus saturated along $x_1\cdots x_n \cdot y_1\cdots y_n$ are reduced schemes, and hence we will simply talk about varieties.}

We first present the \emph{Macaulay2} code that computes the polynomial $p_n$ for arbitrary $n \geq 2$. The number $n$ is set to $2$ in the code but can be changed arbitrarily.
\begin{verbatim}
    clearAll
    n = 2;
    R = QQ[x_1..x_n,y_1..y_n];

    prodx = product toList(x_1..x_n);
    prody = product toList(y_1..y_n);
    sx = sum toList(x_1..x_n);
    sy = sum toList(y_1..y_n);

    for i from 1 to n do (
    P_i = product(toList(y_1..y_(i-1))) * product(toList(y_(i+1)..y_n))
    );

    for i from 1 to n do (
    S_i = sum(toList(x_1..x_(i-1))) + sum(toList(x_(i+1)..x_n))
    );

    p = sum(for i from 1 to n list(P_i*x_i))*sy - n*prody*sx
\end{verbatim}
The reader may like to copy this code into the \href{https://www.unimelb-macaulay2.cloud.edu.au/#home}{web interface for \emph{Macaulay2}} and code along with us in the following case distinction.

\subsection{Singular locus for the petrol variety}\label{sec:singular-locus}
We first test whether $p_2$ defines a radical ideal, so that we do not have to care about schemes as opposed to varieties. The command
\begin{verbatim}
    radical(ideal(p)) == ideal(p)
\end{verbatim}
gives \texttt{true} and, hence, it is indeed a radical ideal. Next we compute the list consisting of the defining ideals of the irreducible components of $V(p_2)$ using the following:
\begin{verbatim}
    L = decompose ideal(p)
\end{verbatim}
It returns the list consisting of the two prime ideals $\left(y_{1}-y_{2}\right)$ and $\left(x_{2}y_{1}-x_{1}y_{2}\right)$ that define the two irreducible components of $V(p_2)$. Note that $(p_2)$ is, in particular, not a prime ideal. The first irreducible component is a linear subspace and is therefore smooth. The second one, also known as the \emph{quadric cone}, has Jacobian matrix $[-y_2,y_1,x_2,-x_1]$ which only has rank $0$ at the origin. So, the second ideal also defines a smooth variety outside the origin. Consequently, the singular locus outside the origin is the intersection of the two irreducible components. The first imposes $y_1 = y_2$ which reduces the second to $y_1(x_2 - x_1)$. Since we are interested in singular points in $(\R^*)^4$, we conclude that $y_1 \neq 0$, so $x_1 = x_2$. Therefore, the singular locus of $V(p_2)$ within $(\R^*)^4$ is $V(x_1 - x_2, y_1 - y_2) = \Delta_2$.\\

 For the cases with $n \geq 3$, we need to use \emph{Macaulay2} more intensively. We first run the defining code for $p_n$ by changing $n = 2$ to one of the values $n = 3, \ldots,n = 9$. The case $n = 10$ does not seem feasible with respect to computation time via computations in \emph{Macaulay2} on a 64-bit Windows machine with an AMD Ryzen 7 PRO 250 CPU
at 3.30 GHz and 16 GB RAM. However, our conjecture is that the singular locus within $(\R^*)^{2n}$ is $\Delta_n$ for arbitrary $n$. Using
\begin{verbatim}
    isPrime ideal(p)
\end{verbatim}
which gives back \texttt{true}, we check that $V(p_n)$ is actually an irreducible variety. This is a first difference: while for $n = 2$, we had two irreducible components, we now have just one. 

The singular locus of $V(p_n)$ is the variety defined by $p_n$ and the entries of the Jacobian matrix of $p_n$ (which is just a vector in this case). So, we compute it with the following command:
\begin{verbatim}
    sing = ideal(p,jacobian(p))
\end{verbatim}
As we are only interested in singularities outside the coordinate hyperplanes, we can remove these from the singular locus, which is algebraically done by performing the \emph{saturation} of the defining ideal of the singular locus along $x_1\cdots x_n\cdot y_1\cdots y_n$. In \emph{Macaulay2}, we conduct this via the following:
\begin{verbatim}
    D = saturate(sing, prodx*prody)
\end{verbatim}
Note that we defined \texttt{prodx} and \texttt{prody} before. In a last step, we define the diagonal $\Delta_n$ via
\begin{verbatim}
    Delta = (
    ideal(for i from 2 to n list(x_1 - x_i)) + 
    ideal(for i from 2 to n list(y_1 - y_i))
    )
\end{verbatim}
and check
\begin{verbatim}
    D == Delta
\end{verbatim}
which gives \texttt{true}. We have therefore shown that, for $n \in \{3,\ldots,9\}$, the variety $V(p_n)$ is irreducible and its singular locus within $(\R^*)^{2n}$ is $\Delta_n$. 

\subsection{RLCT for the petrol variety} We are able to compute the RLCT of the petrol variety at any point in the positive orthant, which are the points relevant in the application, for the cases $n \in \{2,3,4\}$. We do $n=2$ by hand and use further \textit{Macaulay2} calculations for the remaining two cases. The cases $n \geq 5$ seem not to be feasible by our computation on a 64-bit Windows machine with an AMD Ryzen 7 PRO 250 CPU
at 3.30 GHz and 16 GB RAM. Since the singular locus on the positive orthant is just the diagonal $\Delta_n$, our strategy is to blow up along this center.

We start with the case $n=2$. Recall that
\[
p_2 = (x_1y_2 + x_2y_1)(y_1 + y_2) - 2(x_1+x_2)y_1y_2.
\]

\begin{proposition}\label{prop:petrol2}
    Let $a = (a_1,a_2,b_1,b_2) \in \R_{\gt 0}^{4}$. Then
    \[
    \rlct_a(p_2) = \begin{cases}
        (1,2) \text{ if } a \in \Delta_2, \\
        (1,1) \text{ else}.
    \end{cases}
    \]
\end{proposition}

\begin{proof}
    We saw above that $V(p_2)$ is a smooth variety at points $a \notin \Delta_2$ and hence has RLCT equal to $(1,1)$ at such points.
    Now suppose that $a \in \Delta_2$.
    To pull back the diagonal $\Delta_2$ to the subspace defined by $x_2= y_2 = 0$, we apply the linear isomorphism $\sigma$ with $x_2 \mapsto x_2 + x_1$ and $y_2 \mapsto y_2 + y_1$. A simple calculation yields
    \[
    p_2 \circ \sigma =y_2( x_{1}y_{2} -x_{2}y_{1}).
    \]
    We now apply the blow-up $\rho$ along the linear space $C$ defined by $x_2 = y_2 = 0$ and consider $\rho_x$ and $\rho_y$, the two charts corresponding to $x_2$ and $y_2$, respectively. The first one yields
    \[
    p_2 \circ \sigma \circ \rho_x = x_2^2y_2(x_1y_2 - y_1)
    \]
    and has Jacobian determinant $x_2$. The three varieties associated to $x_2$, $y_2$ and $(x_1y_2 - y_1)$ are smooth and meet transversally in all their pairwise intersection points within the fiber $(\sigma \circ \rho_x)^{-1}(a) = \{(a_1,0,b_1,c)  \mid c \in \R_{\gt 0}\}$ that we have to consider by Remark~\ref{remark:facts}. More precisely, $x_2 = 0$ meets $y_2 = 0$ transversally and the intersection point $(a_1,0,b_1,0)$ lies in $(\sigma \circ \rho_x)^{-1}(a)$. Moreover, by a straightforward computation of tangent spaces, $x_2 = 0$ and  $x_1y_2 - y_1 = 0$ meet transversally in their only intersection point $(a_1,0,b_1,b_1/a_1)$ within the fiber. The remaining two varieties do not meet at all. Via Lemma~\ref{lemma:transversal} and Remark~\ref{remark:facts}(4), we derive by considering the monomials $x_2^2$ and $y_2$ that the first component of the RLCT on this chart equals
\[
\lambda = \min \left\{ \frac{1+1}{2}, \frac{0+1}{1}\right\} = 1
\]
with multiplicity $m = 2$.

The $y_2$-chart gives $p_2 \circ \sigma \circ \rho_y = y_2^2(x_1-x_2y_1)$ with Jacobian determinant $y_2$. The two smooth varieties $y_2$ and $x_1-x_2y_1$ meet transversally at their unique intersection point $(a_1,a_1/b_1,b_1,0)$ in $(\sigma \circ \rho_y)^{-1}(a)$ and hence the RLCT on this chart is also $(1,2)$ which makes $\rlct_a(p_2) = (1,2)$.
\end{proof}

The cases $n = 3$ and $n = 4$ behave more uniformely as we will show now.

\begin{proposition}\label{prop:petroln}
    Let $a \in \R_{\gt 0}^{2n}$ with $n \in \{3,4\}$. Then
    \[
    \rlct_a(p_n) = \begin{cases}
        (1,2) \text{ if } a \in \Delta_n, \\
        (1,1) \text{ else}.
    \end{cases}
    \]
\end{proposition}

\begin{proof}
    At points $a \notin \Delta_n$, the variety $V(p_n)$ is smooth as shown above and hence has RLCT equal to $(1,1)$. For $a \in \Delta_n$, we perform computations in \textit{Macaulay2}. We again load the defining code from above for the polynomial $p_n$. We then transform the diagonal $\Delta_n$ to the subspace defined by $x_2 = \ldots = x_n = 0$ and $y_2 = \ldots=y_n = 0$ by applying the map $x_i \mapsto x_i + x_1$ and $y_i \mapsto y_i + y_1$ for $i \neq 1$. Moreover, we transform the union of the coordinate hyperplanes via this map. All of this can be done with the piece of code
    \begin{verbatim}
phi = map(R, R,
    toList(
        apply(1..n, i -> if i == 1 then x_1 else x_i + x_1)
        |
        apply(1..n, i -> if i == 1 then y_1 else y_i + y_1)
        )
    );
p' = phi(p);
S' = phi(prodx*prody);
    \end{verbatim}
Now, we apply the blow-up along the transformation of $\Delta_n$ and consider the chart corresponding to $x_2$, which is done using the following piece of code:
\begin{verbatim}
blowupChartX2 = map(R, R,
    toList(
        apply(1..n, i -> if i <= 2 then x_i else x_2*x_i)
        |
        apply(1..n, i -> if i == 1 then y_i else x_2*y_i)
        )
    );

p'' = blowupChartX2(p');
S'' = blowupChartX2(S');
\end{verbatim}
The command \texttt{factor p''} yields a product of two polynomials of which the first one is $x_2^2$. We extract the second one by defining
\begin{verbatim}
    f = (toList factor p'')#1#0;
\end{verbatim}
Now, we compute the singular locus of $V(f)$, remove its components contained inside the union of the transformed coordinate hyperplanes, which we can exclude as we picked $a$ inside the positive orthant, and compute the dimension of the remaining singular locus:
\begin{verbatim}
    Singf = ideal(f,jacobian(f));

    Df = saturate(Singf, S'');

    dim Df
\end{verbatim}
This gives back dimension $-1$ which is the way that \emph{Macaulay2} expresses that the variety defined by \texttt{Df} is empty. Hence, $V(f)$ is non-singular at points that interest us. Finally, we check that $V(x_2)$ and $V(f)$ meet transversally. For this, we determine the set of those points in $(\R^*)^{2n}$ on which the Jacobian matrix of $(x_2,f)$ has rank smaller $2$, which we show to be the empty set. This is exactly the set of points where all $2\times 2$ minors of the Jacobian matrix vanish. So, after removing the components contained in the coordinate hyperplanes via saturation, the ideal generated by these minors should be equal to the ambient polynomial ring. We check this with
\begin{verbatim}
    saturate(minors(2,jacobian(ideal(x_2,f))),S'')
\end{verbatim}
which gives back \texttt{ideal 1}, the ideal generated by $1$ which is the ambient ring.

So we can handle $f$ as a variable different from $x_2$ by Lemma~\ref{lemma:transversal}. To sum it up, we are given the function $x_2^2f$ and the Jacobian determinant $x_2$ which gives the first component of the RLCT on this chart as
\[
\lambda = \min \left\{ \frac{1+1}{2}, \frac{0+1}{1}  \right\} = 1
\]
with multiplicity $m = 2$. The chart corresponding to $y_2$ works along the same lines by implementing the obvious changes to the code. Since $p_n$ is invariant under permutations of the $x_i$ and the $y_i$, respectively, this completes the proof.
\end{proof}

Based on our insights in Section~\ref{sec:singular-locus} and in Propositions~\ref{prop:petrol2} and~\ref{prop:petroln}, we end our exposition with a conjecture on the values of the RLCT of $p_n$ for arbitrary $n$.

\begin{conj}
    Let $n \geq 2$ be an integer and $a \in \R_{\gt 0}^{2n}$. Then
    \[
    \rlct_a(p_n) = \begin{cases}
        (1,2) \text{ if } a \in \Delta_n, \\
        (1,1) \text{ else}.
    \end{cases}
    \]
\end{conj}

\section*{Acknowledgements}
V.~Fadinger-Held thanks his students for asking the right questions and always keeping him busy.
 D.~Windisch was supported by the FWO grants G0F5921N (Odysseus) and G023721N, and by the KU Leuven grant iBOF/23/064.

\bibliographystyle{amsalpha}
\bibliography{bibliography}
   \end{document}